\documentclass[11pt]{amsart}

\usepackage{amsfonts,epsfig}
\usepackage{latexsym}
\usepackage{amssymb}
\usepackage{amsmath}
\usepackage{amsthm}
\usepackage{graphics}
\usepackage[all]{xy}
\usepackage[T2A]{fontenc}
\usepackage{multirow}
\usepackage{array, booktabs, ctable}
\usepackage{hyperref}
\usepackage{color}
\usepackage{mathtools}
\usepackage{enumerate}
\newtheorem{defn}{Definition}[section]

\newtheorem{lemma}[defn]{Lemma}

\newtheorem{theorem}[defn]{Theorem}

\theoremstyle{definition}
\newtheorem*{remark}{Remark}

\newcommand{\lmfdbec}[3]{\href{http://www.lmfdb.org/EllipticCurve/Q/#1#2#3}{{\text{\rm#1#2#3}}}}

\newcommand{\Q}{\mathbb Q}
\newcommand{\Z}{\mathbb Z}

\newcommand{\Gal}{\operatorname{Gal}}

\begin{document}

\bibliographystyle{plain}

\title[Torsion Cyclotomic Extension]{Torsion of Rational Elliptic Curves over the $\mathbb{Z}_p$-Extensions of Quadratic Fields}

\author{Omer Avci}

\address{Dept. of Mathematics, Bogazici University, Istanbul, Turkey}
\email{omeravci372742@gmail.com} 



\begin{abstract} Let $E$ be an elliptic curve defined over $\mathbb{Q}$.
For a quadratic number field $K$ and an odd prime
number $p$, let $L$ be a $\Z_p$-extension of $K$. We prove that $E(L)_{\text{tors}}=E(K)_{\text{tors}}$ when $p>5$. It enables us to classify the groups that can be realized as the torsion subgroup $E(L)_{\text{tors}}$, by using the classification of torsion subgroups over the quadratic fields. 

\end{abstract}

\maketitle

\section{Introduction and Notation}

Let $L$ denote a number field, and let $E$ be an elliptic curve over $L$. The Mordell-Weil theorem states that the set of $L$-rational points on $E$ forms a finitely generated abelian group. Specifically, denoting $E(L)$ as the set of $L$-rational points on $E$, we have
$$
E(L) \cong E(L)_{\text{tors}} \oplus \mathbb{Z}^r.
$$
Here, $r$ is a nonnegative integer called the rank, and $E(L)_{\text{tors}}$ is a finite group known as the torsion subgroup of $E$ over $L$.

The classification of $E(L)_{\text{tors}}$ is a well-studied and compelling 
problem, with variations depending on whether one considers all elliptic curves
defined over $L$, or only those defined over $\mathbb{Q}$. A foundational 
result in this area is Mazur’s theorem, which gives a complete classification 
when $L = \mathbb{Q}$. Extending this work, Najman determined the possible 
torsion structures of rational elliptic curves over quadratic and cubic fields 
in \cite{NajmanQuadratic}.

A natural extension of this problem is to classify $E(L)_{\text{tors}}$, when 
the extension degree $[L:\Q]$ is not finite. For certain extensions $L$, the 
torsion subgroup remains finite, and all possible torsion structures can still 
be classified. As an example Chou classified $E(\Q^{ab})_{\text{tors}}$ in 
\cite{Chou}, where $\Q^{ab}$ is the maximal abelian extension of $\Q$.

Another example is the classification of $E(L)_{\text{tors}}$, where $L$ is a $\Z_p$-extension of a number field. 
Chou et al. determined the classification of $E(\Q_{\infty,p})_{\text{tors}}$ in 
\cite{Chou-p-adic}, where $\Q_{\infty,p}$ is the unique $\Z_p$-extension of 
$\Q$. Moreover, they proved $E(\Q_{\infty,p})_{\text{tors}} = E(\Q)_{\text{tors}}$ if $p\geq 5$.

In this paper, we prove an analogue of the result in \cite{Chou-p-adic} by considering the $\Z_p$-extensions of quadratic fields. Here, we present our main theorem.

\begin{theorem}\label{maintheorem}
Let $E/\Q$ be an elliptic curve. Let $K$ be a quadratic field. Let $p>5$ be a prime. Let $L$ be a $\Z_p$-extension of $K$. Then, $E(L)_{\text{tors}}= E(K)_{\text{tors}}$.
\end{theorem}

If $K$ is real quadratic field, it is guaranteed that the cyclotomic $\Z_p$-extension exist, and it may or may not be the unique $\Z_p$-extension of $K$.
If $K$ is imaginary quadratic field, it is guaranteed that there are infinitely many $\Z_p$-extensions, as shown in \cite{iwasawa}.
Among the $\Z_p$-extensions of an imaginary quadratic field $K$, there are two important fields,
 $K_{\text{cyc}}$ and $K_{\text{anti}}$, which denotes the cyclotomic and anti-cyclotomic $\Z_p$-extensions of $K$ respectively, where $\Gal(K_{cyc}/\Q)$ is pro-cyclic and $\Gal(K_{anti}/\Q)$ is pro-dihedral \cite{anticyclo}. Moreover, any $\Z_p$-extension $L$ of $K$ is contained in $K_{cyc}K_{anti}$.

We will review relevant results from the literature in Section \ref{literaturesection}. In Section \ref{eliminationsection}, we will introduce techniques developed for eliminating group structures that cannot be realized as $E(L)_{\text{tors}}$, where $L$ is any Galois extension of $\mathbb{Q}$. Finally, in Section \ref{edgecases}, we will focus on the torsion subgroups $E(K_{cyc})_{\text{tors}}$ and $E(K_{anti})_{\text{tors}}$ for the cases $p = 3$ and $p = 5$.

\noindent \textbf{Acknowledgements.} I would like to thank Antonio Lei for bringing this problem to my attention. I am also grateful to Payman Eskandari for his valuable comments and insightful discussions.

\section{Torsion of Rational Elliptic Curves} \label{literaturesection}
In this section, we will present some important results from the literature about the classification of $E(L)_{\text{tors}}$ and rational isogenies of the elliptic curves, where $L/\Q$ is a Galois extension and $E/\Q$ is an elliptic curve.

\begin{theorem}[Mazur, \cite{Mazur}]
   Let $E/\mathbb{Q}$ be an elliptic curve. Then
  $$E(\mathbb{Q})_{\text{tors}} \simeq 
  \begin{cases}
      \mathbb{Z} / N \mathbb{Z} & \text{ with } 1\leq N\leq 10 \text{ or } N=12 \text{ or},  \\
           \mathbb{Z} / 2 \mathbb{Z} \times      \mathbb{Z} /2 N \mathbb{Z} & \text{ with } 1\leq N\leq 4.
  \end{cases}  $$ 
  All of the 15 torsion subgroups appear infinitely often.
\end{theorem}

\begin{theorem}(Najman, \cite{NajmanQuadratic})
\label{najmanquadratic}
  Let $E/\mathbb{Q}$ be an elliptic curve, and let $K$
be a quadratic number field. Then
$$E(K)_{\text{tors}} \simeq 
  \begin{cases}
      \mathbb{Z} / N \mathbb{Z} & \text{ with } 1\leq N\leq 10 \text{ or } N=12,15,16,\text{ or}\\
           \mathbb{Z} / 2 \mathbb{Z} \times      \mathbb{Z} /2 N \mathbb{Z} & \text{ with } 1\leq N\leq 6 \text{ or}\\
           \mathbb{Z} /3  \mathbb{Z}  \times    \mathbb{Z} /3 N \mathbb{Z} & \text{ with } N=1,2, \text{ only if } K=\mathbb{Q}(\sqrt{-3}) \text{ or}\\
              \mathbb{Z} /4  \mathbb{Z} \times    \mathbb{Z} /4 \mathbb{Z} & \text{ only if } K=\mathbb{Q}(\sqrt{-1}).        
  \end{cases}  $$
  Each of these groups, except for $\mathbb{Z}/15\mathbb{Z}$, appears as the torsion structure over a quadratic field for infinitely many pairs of rational elliptic curves $E$ and quadratic field $K$. The elliptic curves \lmfdbec{50}{b}{1} and \lmfdbec{50}{a}{3} have $15$-torsion over $\mathbb{Q}(\sqrt{5})$, \lmfdbec{50}{b}{2} and \lmfdbec{450}{b}{4} have $15$-torsion over $\mathbb{Q}(\sqrt{-15})$. These are the only
rational curves having non-trivial $15$-torsion over any quadratic field.
\end{theorem}

\begin{theorem}(Chou, \cite{Chou}, Theorem 1.2)\label{overmaximalabelian}
      Let $E/\mathbb{Q}$ be an elliptic curve. 
    Then $E(\Q^{ab})_{\text{tors}}$ is isomorphic to one of the following
groups:
$$E(\Q^{ab})_{\text{tors}} \simeq 
  \begin{cases}
      \mathbb{Z} / N \mathbb{Z} &  1, 3, 5, 7, 9, 11, 13, 15, 17, 19, 21, 25, 27, 37, 43, 67, 163, \\
           \mathbb{Z} / 2 \mathbb{Z} \times      \mathbb{Z} /2 N \mathbb{Z} & \text{ with } 1\leq N\leq 9,\\
           \mathbb{Z} /3  \mathbb{Z}  \times    \mathbb{Z} /3 N \mathbb{Z} & \text{ with } N=1,3, \\
              \mathbb{Z} /4  \mathbb{Z} \times    \mathbb{Z} /4N \mathbb{Z} &\text{ with } N=1,2,3,4, \\
                 \mathbb{Z} /5  \mathbb{Z} \times    \mathbb{Z} /5 \mathbb{Z}\\
                    \mathbb{Z} /6  \mathbb{Z} \times    \mathbb{Z} /6 \mathbb{Z}\\
                       \mathbb{Z} /8  \mathbb{Z} \times    \mathbb{Z} /8 \mathbb{Z}
  \end{cases}  $$
\end{theorem}

\begin{theorem}(Chou, \cite{Chou}, Corollary to Theorem 1.2)\label{allgroupslist}
      Let $E/\mathbb{Q}$ be an elliptic curve. Let $L$ be an abelian Galois extension of $\Q$.
    Then $E(L)_{\text{tors}}$ is isomorphic to one of the following
groups:
$$E(L)_{\text{tors}} \simeq 
  \begin{cases}
      \mathbb{Z} / N \mathbb{Z} & \text{ with } 1\leq N\leq 19, \text{ or } N=21,25,27,37,43,67,163, \\
           \mathbb{Z} / 2 \mathbb{Z} \times      \mathbb{Z} /2 N \mathbb{Z} & \text{ with } 1\leq N\leq 9,\\
           \mathbb{Z} /3  \mathbb{Z}  \times    \mathbb{Z} /3 N \mathbb{Z} & \text{ with } N=1,2,3, \\
              \mathbb{Z} /4  \mathbb{Z} \times    \mathbb{Z} /4N \mathbb{Z} &\text{ with } N=1,2,3,4, \\
                 \mathbb{Z} /5  \mathbb{Z} \times    \mathbb{Z} /5 \mathbb{Z}\\
                    \mathbb{Z} /6  \mathbb{Z} \times    \mathbb{Z} /6 \mathbb{Z}\\
                       \mathbb{Z} /8  \mathbb{Z} \times    \mathbb{Z} /8 \mathbb{Z}
  \end{cases}  $$
\end{theorem}

\begin{theorem}[Fricke, Kenku, Klein, Kubert, Ligozat, Mazur, and Ogg, among others]\label{rationalisogeny}
If $E/\Q$ has an $n$-isogeny, $n \leq 19 $ or $n \in \{21, 25, 27, 37, 43, 67, 163\}$. If E does not have complex multiplication, then $n \leq 18$ or $n \in \{21, 25, 37\}$.
\end{theorem}

For $n \in \{11, 15, 17, 19, 21, 27, 37, 43, 67, 163\}$, the modular curve $X_0(n)$ has finitely many rational points. Consequently, there are only finitely many elliptic curves defined over $\Q$, or more precisely, finitely many rational $j$-invariants, with an $n$-isogeny. The corresponding list of $j$-invariants can be found in Table 2 of Section 7 in \cite{Chou} and Table 4 of Section 9 in \cite{fieldofdefinition}.

\begin{lemma}[Chou, \cite{Chou}]\label{n-isogeny}
    Let $L$ be a Galois extension of $\Q$, and let $E$ be an elliptic curve over $\Q$. If $E(L)_{\text{tors}}\cong \Z/m\Z \times \Z/mn\Z$, then $E$ has an $n$-isogeny over $\Q$.
\end{lemma}

\section{Eliminating Possible Torsion}\label{eliminationsection}

In this section, we present our results on the classification of $E(L)_{\text{tors}}$, where $L/\Q$ is a Galois extension and $E/\Q$ is an elliptic curve. We will also provide the proof of our main theorem at the conclusion of this section.

\begin{lemma}\label{finitetorsion}
Let $E/\Q$ be an elliptic curve. 
  Let $L$ be a Galois extension of $\Q$. Let $n\geq 2$ be the largest positive integer such that $\mu_n\subset L$. Then, $|E(L)_{\text{tors}}|\leq 163n^2.$  
\end{lemma}

\begin{proof}
   Let $K \subset L$ be any finite extension of $\mathbb{Q}$ of degree $[K:\mathbb{Q}] < \infty$. Then, $E(K)_{\text{tors}}$ is finite and has the form $E(K)_{\text{tors}} \cong \mathbb{Z}/a\mathbb{Z} \times \mathbb{Z}/ab\mathbb{Z}$ for some positive integers $a$ and $b$. By the Weil pairing, we must have $\mu_a \subset K$. From our assumption about $L$, it follows that $a \leq n$. Moreover, by Lemma~\ref{n-isogeny}, $E$ has a rational $b$-isogeny. Then, by Theorem~\ref{rationalisogeny}, we obtain $b \leq 163$. As a result, we conclude that $|E(K)_{\text{tors}}| \leq 163n^2$.
Since this bound is uniform for all such subfields $K$ of $L$, we deduce that
$|E(L)_{\text{tors}}| \leq 163n^2.$
\end{proof}

\begin{lemma}[A., Lemma 3.1, \cite{Omer}] \label{4pairinggeneralmore}
  Let $E/\Q$ be an elliptic curve. Let $L$ be a Galois number field of degree $n$.  Let $q$ be an odd prime power. Let $E(L)[q] \cong \Z/q\Z$. If $\gcd(n,\phi(q))=1$ then $E(L)[q] = E(\Q)[q].$ If $\gcd(n,\phi(q))=2$ then $E(L)[q] = E(\Q(\sqrt{d}))[q]$ for some square-free integer $d$ satisfying $\sqrt{d} \in L$.  
\end{lemma}

\begin{lemma}\label{2torsionlemma4}
     Let $L$ be a finite Galois extension of $\mathbb{Q}$ with $[L:\mathbb{Q}]=n$ for some positive integer $n$ not divisible by $4$. Let $E/\mathbb{Q}$ be an elliptic curve. Then, there is a square-free integer $d$ satisfying $\sqrt{d}\in L$ and the following:
     \begin{enumerate}[(i)]
         \item If $E(L)[2^\infty]\cong \Z/2^k\Z$ for some $m\geq 0$, then $E(L)[2^\infty]=E(\Q(\sqrt{d}))[2^\infty]$.
         \item If $E(L)[4]\cong \Z/2\Z\times \Z/4\Z$, then $E(L)[4]=E(\Q(\sqrt{d}))[4]$. 
         \item If $E(L)[8]\cong \Z/2\Z\times \Z/8\Z$, then $E(L)[2^\infty]=E(L)[8]=E(\Q(\sqrt{d}))[8]$.
     \end{enumerate}
\end{lemma}

\begin{proof}
If $n$ is even, then let $K$ be the unique quadratic subfield of $L$, and otherwise let $K = \mathbb{Q}$. Let us deal with the first statement, as the other two are relevant later.

Assume that the elliptic curve has the equation
\begin{equation*}
E \colon y^2 = x^3 + Ax + B
\end{equation*}
for some $A, B \in \mathbb{Q}$. If the cubic $x^3 + Ax + B$ has no rational root, then its splitting field either lies in $L$ or intersects $\mathbb{Q}$ trivially with $L$. This would imply that either all of the $2$-torsion lies in $E(L)_{\text{tors}}$, or none of it does. This contradicts the assumption that $E(L)[2] \cong \mathbb{Z}/2\mathbb{Z}$. Hence, we must have $E(L)[2] = E(\mathbb{Q})[2]$.

Now, we will show that for any $k \geq 1$, we have
\begin{equation*}
E(L)[2^k] = E(K)[2^k].
\end{equation*}

There are $2^{k-1}$ points of order $2^k$ in $E(L)_{\text{tors}}$, yielding $2^{k-2}$ distinct $x$-coordinates. Consider the $\mathrm{Gal}(L/\mathbb{Q})$ orbits of these $x$-coordinates, and let $s$ be the smallest orbit size. Pick a point $P = (x, y) \in E(L)_{\text{tors}}$ of order $2^k$ such that
\begin{equation*}
|\{x^\sigma : \sigma \in \mathrm{Gal}(L/\mathbb{Q})\}| = s.
\end{equation*}
Then $[\mathbb{Q}(x):\mathbb{Q}] = s$, and since $[\mathbb{Q}(x, y):\mathbb{Q}(x)] \leq 2$, we get
\begin{equation*}
[\mathbb{Q}(P):\mathbb{Q}] = s \text{ or } 2s.
\end{equation*}
As $P$ generates $E(L)[2^k]$, any other point of order $2^k$ lies in $\mathbb{Q}(P)$, so each orbit size divides $2s$. Since $s$ is minimal, all orbit sizes are $s$ or $2s$, and they partition the $2^{k-2}$ distinct $x$-coordinates. Therefore, $s$ divides $2^{k-2}$ and must be a power of $2$.

Moreover, since $s$ divides $n$ and $n$ is not divisible by $4$, we must have $s = 1$ or $2$. Hence, $x \in K$, and since $\mathbb{Q}(P)$ cannot be a quartic field, $P \in E(K)_{\text{tors}}$. Thus, $E(L)[2^k] = E(K)[2^k]$.

Now consider the second statement. There are $4$ points of order $4$ in $E(L)_{\text{tors}}$, giving $2$ distinct $x$-coordinates. These $x$-coordinates are permuted by $\mathrm{Gal}(L/\mathbb{Q})$, so they lie in $K$. Let $P = (x, y) \in E(L)[4]$ be one such point. Then $[\mathbb{Q}(x, y):\mathbb{Q}(x)] \leq 2$, so $P \in E(K)_{\text{tors}}$. Since $2P$ has order $2$ and lies in $E(K)_{\text{tors}}$, the cubic $x^3 + Ax + B$ must have a root in $K$, implying all its roots lie in $K$. Thus,
\begin{equation*}
E(K)[2] \cong \mathbb{Z}/2\mathbb{Z} \times \mathbb{Z}/2\mathbb{Z}.
\end{equation*}
Together with $P$, this generates $E(L)[4]$, so $E(L)[4] = E(K)[4]$.

Now, for the third statement, consider the $8$ points of order $8$ in $E(L)_{\text{tors}}$, corresponding to $4$ distinct $x$-coordinates. Partition these into $\mathrm{Gal}(L/\mathbb{Q})$ orbits. Since $n$ is not divisible by $4$, there cannot be a single orbit of size $4$, so at least one orbit has size at most $2$.

Let $P = (x, y)$ be a point of order $8$ such that
\begin{equation*}
|\{x^\sigma : \sigma \in \mathrm{Gal}(L/\mathbb{Q})\}| \leq 2.
\end{equation*}
Then $x \in K$, and since $[\mathbb{Q}(x, y):\mathbb{Q}(x)] \leq 2$, also $y \in K$, so $P \in E(K)_{\text{tors}}$. Since all $2$-torsion lies in $K$ and together with $P$ these generate $E(L)[8]$, we conclude that $E(L)[8] = E(K)[8]$.

To complete the proof, consider points of order $16$ in $E(L)_{\text{tors}}$. If $E(L)[2^\infty] \neq E(L)[8]$, then
\begin{equation*}
E(L)[16] \cong \mathbb{Z}/2\mathbb{Z} \times \mathbb{Z}/16\mathbb{Z}.
\end{equation*}
There are $16$ such points, giving $8$ distinct $x$-coordinates. Analyzing $\mathrm{Gal}(L/\mathbb{Q})$ orbits, there must be an orbit of size at most $2$ unless the orbits are of size $3$ and $5$. In the first case, we use the same reasoning as before to conclude $E(L)[16] = E(K)[16]$, contradicting Theorem~\ref{najmanquadratic}.

Now suppose there are two orbits of sizes $3$ and $5$. Pick $P = (x, y) \in E(L)_{\text{tors}}$ of order $16$ such that
\begin{equation*}
|\{x^\sigma : \sigma \in \mathrm{Gal}(L/\mathbb{Q})\}| = 3.
\end{equation*}
Then $[\mathbb{Q}(x):\mathbb{Q}] = 3$ and $[\mathbb{Q}(P):\mathbb{Q}] = 3$ or $6$. Since $P$ and the $2$-torsion generate $E(L)[16]$, the entire torsion subgroup is defined over a field of degree at most $6$, implying that all other $x$-coordinates also lie in this field. But then the other orbit cannot have size $5$, leading to a contradiction. Therefore, such an orbit structure is impossible.

This completes the proof of the third statement.
\end{proof}

\begin{lemma}\label{2torsionlemma34}
      Let $L$ be a finite Galois extension of $\mathbb{Q}$ with $[L:\mathbb{Q}]=n$ for some positive integer $n$ not divisible by $3$ and $4$. Let $E/\mathbb{Q}$ be an elliptic curve. Then, $E(L)[2^\infty]$ can be the following:
     \begin{equation*}
        E(L)[2^\infty] \cong \begin{cases}
              \mathbb{Z} / 2^k \mathbb{Z} & \text{ with } 0\leq k\leq 4, \text{ or } \\
           \mathbb{Z} / 2 \mathbb{Z} \times      \mathbb{Z} /2^k \mathbb{Z} & \text{ with } 1\leq k\leq 3, \text{ or}\\
              \mathbb{Z} /4  \mathbb{Z} \times    \mathbb{Z} /4 \mathbb{Z} & \text{ only if } \mathbb{Q}(\sqrt{-1})\subset L.    
        \end{cases} 
     \end{equation*}
     Moreover, there is a square-free integer $d$ satisfying $\sqrt{d}\in L$ and $E(L)[2^\infty]=E(\Q(\sqrt{d}))[2^\infty]$.
\end{lemma}

\begin{proof}
Suppose $E(L)[2^\infty] \cong \mathbb{Z}/2^a\mathbb{Z} \times \mathbb{Z}/2^b\mathbb{Z}$ for some integers $0 \leq a \leq b$. We consider three cases depending on the value of $a$.

If $n$ is even, let $K$ be the unique quadratic subfield of $L$. Otherwise, define $K := \mathbb{Q}$.

From the Weil pairing, we know that $\mu_{2^a} \subseteq L$. Since $4 \nmid n$, this implies $a \leq 2$.

\textbf{Case 1: $a = 0$.} Then $E$ has a rational $2^b$-isogeny. By Theorem~\ref{rationalisogeny}, we must have $b \leq 4$, since $E$ cannot have a rational $32$-isogeny.

Furthermore, Lemma~\ref{2torsionlemma4} gives
\begin{equation*}
E(L)[2^\infty] = E(K)[2^\infty],
\end{equation*}
because $4 \nmid n$ suffices for this equality.

\textbf{Case 2: $a = 1$.} If $b > 1$, we may again apply Lemma~\ref{2torsionlemma4} to conclude
\begin{equation*}
E(L)[2^\infty] = E(K)[2^\infty].
\end{equation*}
Then, Theorem~\ref{najmanquadratic} implies $b \leq 3$.

\textbf{Case 3: $a = 2$.} In this case, the Galois group
\begin{equation*}
\Gal(\mathbb{Q}(E[4])/\mathbb{Q}) \hookrightarrow \mathrm{GL}_2(\mathbb{Z}/4\mathbb{Z}),
\end{equation*}
which has order $96$. Therefore,
$[\mathbb{Q}(E[4]) : \mathbb{Q}] \mid 96$.

Since $\mathbb{Q}(E[4]) \subseteq L$ and $[L:\mathbb{Q}] = n$, we also have
$[\mathbb{Q}(E[4]) : \mathbb{Q}] \mid n$.
As $\gcd(n, 96) \leq 2$, it follows that
$[\mathbb{Q}(E[4]) : \mathbb{Q}] \leq 2$,
so $\mathbb{Q}(E[4]) \subseteq K$.

Moreover, since $\mu_4 \subseteq L$ by the Weil pairing, we find $\mathbb{Q}(\sqrt{-1}) \subseteq L$, and therefore $K = \mathbb{Q}(\sqrt{-1})$. Since $\mathbb{Q}(E[4]) \ne \mathbb{Q}$, we conclude
\begin{equation*}
\mathbb{Q}(E[4]) = \mathbb{Q}(\sqrt{-1}) = K.
\end{equation*}

Now we claim that $b = 2$. Suppose, for contradiction, that $b \geq 3$. Then
\begin{equation*}
E(L)[8] \cong \mathbb{Z}/4\mathbb{Z} \times \mathbb{Z}/8\mathbb{Z}.
\end{equation*}
We will show that $E(L)[8] = E(K)[8]$, which contradicts Theorem~\ref{najmanquadratic}.

We already have $E(L)[4] = E(K)[4]$. Now consider the points of order $8$ in $E(L)[8]$: there are $16$ such points, corresponding to $8$ distinct $x$-coordinates. By analyzing the $\Gal(L/\mathbb{Q})$-orbit of these $x$-coordinates, we find that at least one such point must be defined over $K$. This point, combined with the $4$-torsion, generates $E(L)[8]$, hence
\begin{equation*}
E(L)[8] = E(K)[8],
\end{equation*}
a contradiction.

Therefore, $b = 2$, and we conclude that
\begin{equation*}
E(L)[2^\infty] = E(L)[4] = E(K)[4] = E(K)[2^\infty]
\end{equation*}
as desired. 
\end{proof}

\begin{lemma}\label{3torsionlemma}
     Let $L$ be a finite Galois extension of $\mathbb{Q}$ with $[L:\mathbb{Q}]=n$ for some positive integer $n$ not divisible by $3$ and $4$. Let $E/\mathbb{Q}$ be an elliptic curve. Then, $E(L)[2^\infty]$ can be the following:
     \begin{equation*}
        E(L)[2^\infty] \cong \begin{cases}
              \mathbb{Z} / 3^k \mathbb{Z} & \text{ with } 0\leq k\leq 2, \text{ or } \\
              \mathbb{Z} /3  \mathbb{Z} \times    \mathbb{Z} /3 \mathbb{Z} & \text{ only if } \mathbb{Q}(\sqrt{-3})\subset L.    
        \end{cases} 
     \end{equation*}
     Moreover, there is a square-free integer $d$ satisfying $\sqrt{d}\in L$ and $E(L)[3^\infty]=E(\Q(\sqrt{d}))[3^\infty]$. 
\end{lemma}
\begin{proof}
Suppose $E(L)[3^\infty] \cong \mathbb{Z}/3^a\mathbb{Z} \times \mathbb{Z}/3^b\mathbb{Z}$ for some integers $0 \leq a \leq b$. We consider two cases depending on whether $n$ is even or odd.

If $n$ is even, let $K$ be the unique quadratic subfield of $L$. Otherwise, put $K= \mathbb{Q}$.

From the Weil pairing, we know that $\mu_{3^a} \subseteq L$. Since $3 \nmid n$, this implies $a \leq 1$.

\textbf{Case 1: $a = 0$.} In this case, $E$ admits a rational $3^b$-isogeny. By Theorem~\ref{rationalisogeny}, we must have $b \leq 3$, as there is no rational $81$-isogeny.

Furthermore, Lemma~\ref{4pairinggeneralmore} implies that $E(L)[3^\infty] = E(K)[3^\infty]$. Then, applying Theorem~\ref{najmanquadratic}, we rule out the possibility that $b = 3$. So the only remaining options are $b \leq 2$.

\textbf{Case 2: $a = 1$.} In this case, $\mathbb{Q}(E[3])$ is a Galois extension of $\mathbb{Q}$ with Galois group contained in $\mathrm{GL}_2(\mathbb{Z}/3\mathbb{Z})$, which has order $48$. Therefore,
\begin{equation*}
    [\mathbb{Q}(E[3]) : \mathbb{Q}] \mid 48.
\end{equation*}
Since $\mathbb{Q}(E[3]) \subseteq L$ and $[L:\mathbb{Q}] = n$, we also have $[\mathbb{Q}(E[3]) : \mathbb{Q}] \mid n$. But $\gcd(n, 48) \leq 2$ by assumption, so $[\mathbb{Q}(E[3]) : \mathbb{Q}] \leq 2$.

This shows that $\mathbb{Q}(E[3]) \subseteq K$. On the other hand, by the Weil pairing, $\mathbb{Q}(\sqrt{-3}) \subseteq L$, so in this case we must have $K = \mathbb{Q}(\sqrt{-3})$. Since $\mathbb{Q}(E[3])$ cannot be $\mathbb{Q}$, it follows that
\begin{equation*}
    \mathbb{Q}(E[3]) = \mathbb{Q}(\sqrt{-3}) = K.
\end{equation*}

Now we claim that $b = 1$. Suppose for contradiction that $b \geq 2$. Then
\begin{equation*}
    E(L)[9] \cong \mathbb{Z}/3\mathbb{Z} \times \mathbb{Z}/9\mathbb{Z}.
\end{equation*}
We will show that $E(L)[9] = E(K)[9]$, which contradicts Theorem~\ref{najmanquadratic}.

We have already established that $E(L)[3] = E(K)[3]$. Consider the points of order $9$ in $E(L)[9]$: there are 18 such points, which yield 9 distinct $x$-coordinates. By considering the action of $\mathrm{Gal}(L/\mathbb{Q})$ on these $x$-coordinates, we can find a point of order 9 whose $x$-coordinate is fixed under $\mathrm{Gal}(L/K)$. Hence, such a point lies in $E(K)[9]$.

Combining this point with the $3$-torsion, we obtain all of $E(L)[9]$ inside $E(K)[9]$, a contradiction. Thus, $b$ cannot be greater than 1.

We conclude that
\begin{equation*}
    E(L)[3^\infty] = E(L)[3] = E(K)[3] = E(K)[3^\infty],
\end{equation*}
as desired.
\end{proof}

\begin{lemma}\label{anticyclomain}
    Let $E/\Q$ be an elliptic curve. Let $K=\Q(\sqrt{-d})$ be an imaginary quadratic
    field where $d$ is a square-free positive integer. Let $p$ be an odd prime and
    $K_{anti}$ be the anti-cyclotomic $\Z_p$-extension of $K$. If $E(K_{anti})[N]\cong \Z/N\Z$ for some odd positive integer $N$, then $E(K_{anti})[N] = E(K)[N]$.
\end{lemma}

\begin{proof}
    Showing that the assumption holds for a prime power $q$ such that $\gcd(q,N/q)=1$ will be sufficient.

    Let $\tau$ denote complex conjugation. By Main Lemma 1\textit{(ii)} from \cite{anticyclo}, we know that $\tau$ operates on $\Gal(K_{anti}/K)$ by inversion, that is
    \begin{equation*}
        \tau^{-1}\sigma\tau=\sigma^{-1},\quad \text{for any $\sigma\in \Gal(K_{anti}/K)$}.
    \end{equation*}
    Since $\tau=\tau^{-1}$, we can observe that $(\tau\sigma)^2=1$, so $\tau\sigma$ has order at most 2 where $\tau\sigma\in \Gal(K_{anti}/\Q)$.

    Now, let $P=(x,y)\in E(K_{anti})[q]$ be a point of order $q$. Observe that
    \begin{equation*}
        P^{\tau\sigma}=aP,\quad \text{where $a\in(\Z/q\Z)^\times$}.
    \end{equation*}
    Then we have
    \begin{equation*}
        P=P^{(\tau\sigma)^2}=a^2P
    \end{equation*}
    which implies that $a\equiv \pm1\pmod{q}$. So we find 
    \begin{equation*}
        P^{\tau\sigma}=\pm P=(x,\pm y).
    \end{equation*}
    
    Observe that $x^{\tau\sigma}=x$ holds for any $\sigma\in \Gal(K_{anti}/K)$, so if $\sigma=id$ then $x^\tau=x$. So we have
    \begin{equation*}
        x^\sigma=x^\tau=x, \quad  \text{for any $\sigma\in \Gal(K_{anti}/K)$}.
    \end{equation*}
    Hence $x\in K$.

    Finally, let the elliptic curve $E$ has the equation $E:y^2=x^3+Ax+B$ for some $A,B\in \Q$. Then $y$ is algebraic over $\Q[x]$ with degree either 1 or 2. Note that $K_{anti}$ has no quartic subfield. This ensures that $y\in K$. Hence $P\in E(K)[q]$.
\end{proof}

\begin{lemma}\label{rootsofunityinZpextensions}
    Let $K=\Q(\sqrt{d})$ be a quadratic
    field where $d$ is a square-free integer. Let $p>3$ be a prime. Let $L$ be a $\Z_p$-extension of $K$. Let $n$ be a positive integer satisfying $\mu_n\subset L$. Then, $\mu_n\subset K$ holds.
\end{lemma}

\begin{proof}
    Assume that $\mu_n \subset L$ for some positive integer $n$. It is clear that $n = 1, 2$ always work. We know that $K$ is the unique quadratic subfield of $L$. If $n = 4$, then $K = \Q(\sqrt{-1})$ must hold. If $n = 3, 6$, then $K = \Q(\sqrt{-3})$ must hold. In any case, it is clear that $n = 8$ or $n = 9$ does not work because $L$ does not have a quartic or cubic subfield. 

    Assume that $n \neq 1, 2, 3, 4, 6$. Then $n$ must have a prime factor $q \neq 2, 3$. In this case, $d = q$ and $q \equiv 1 \pmod{4}$, or $d = -q$ and $q \equiv 3 \pmod{4}$ must hold. Moreover, $\Q(\zeta_q)$ is a subfield of $L$, with extension degree $(q-1)/2$ over $K$. Then, $(q-1)/2$ must be a power of $p$, since $L$ is a $\Z_p$-extension of $K$. Since $q > 3$, we can see that $p$ must divide $q - 1$, therefore $p < q$. We know that $q$ is ramified in the extension $\Q(\zeta_q)/K$, thus $q$ is ramified in $L/K$. In Iwasawa \cite{iwasawa}, it is shown that any $\Z_p$-extension of $K$ is unramified outside $p$. Thus, we get the desired contradiction.
\end{proof}

Now, we proceed to the proof of our main theorem. 

\begin{proof}[Proof of Theorem \ref{maintheorem}]

    In Lemma \ref{rootsofunityinZpextensions}, we have shown that any $\Z_p$-extension $L$ contains only finitely many roots of unity. Then, by Lemma \ref{finitetorsion}, we know that $E(L)_{\text{tors}}$ is finite. Therefore, there exist positive integers $m, n$ satisfying $E(L)_{\text{tors}} \cong \Z/n\Z \times \Z/mn\Z$. By the Weil pairing, this would imply $\mu_n \subset L$. Again, by Lemma \ref{rootsofunityinZpextensions}, we can see that $n = 1, 2, 3, 4, 6$ must hold. On the other hand, by Lemma \ref{n-isogeny}, we can see that the elliptic curve $E$ has a rational $m$-isogeny. Using Theorem \ref{rationalisogeny}, we can conclude that $m \leq 19$ or $m \in \{21, 25, 27, 37, 43, 67, 163\}$. As a result, if $E(L)[q]$ is not trivial for a prime $q$, then $q$ must be one of the primes in the set $\{2, 3, 5, 7, 11, 13, 17, 19, 37, 43, 67, 163\}$.

    If we can show that $E(L)[q^\infty] = E(K)[q^\infty]$ for the primes in this list, then we are done. We are dealing with an infinite extension, but if we can get the same result for all finite subfields of $L$, then we would achieve our goal. It is clear that any finite subfield of $L$ has the extension degree $p^a$ or $2 \cdot p^a$ for some non-negative integer $a$. Such extension degrees are not divisible by $3$ and $4$ due to our assumption $p > 5$. Then, by Lemma \ref{2torsionlemma34}, we obtain $E(L)[2^\infty] = E(K)[2^\infty]$, and by Lemma \ref{3torsionlemma}, we obtain $E(L)[3^\infty] = E(K)[3^\infty]$.

    For any prime $q > 3$, if $E(L)[q]$ is not trivial, then $E(L)[q] \cong \Z/q\Z$ must hold, otherwise we would have $\mu_q \subset L$ due to the Weil pairing, and we already discussed that this is not possible. Moreover, if $E(L)[q^\infty] \cong \Z/q^k\Z$ for some $k > 1$, then $q = 5$ must hold because it is the only case where $E$ can have a rational $q^2$-isogeny for some prime $q > 3$, due to Theorem \ref{rationalisogeny}. For $q = 5$, we can see that the extension degree of any subfield of $L$ is not divisible by $4$ and $5$ due to our assumption $p > 5$. Thus, we can use Lemma \ref{4pairinggeneralmore} to obtain $E(L)[5^\infty] = E(K)[5^\infty]$.

    If any prime $q > 5$ satisfies $E(L)[q] \cong \Z/q\Z$, then we have $E(L)[q^\infty] = E(L)[q]$, so let us show that $E(L)[q] = E(K)[q]$. Let us look at all the possible prime factors of $q - 1$ for the primes $q > 5$ in the set we have given above. It is easy to see that the set of prime factors of $q - 1$ is equal to $\{2, 3, 5, 7, 11\}$. Then, in the case $p > 11$, we are done by using Lemma \ref{4pairinggeneralmore} because we get $\gcd(2 \cdot p^a, q - 1) = 2$ for all $p > 11$ and $q \in \{7, 11, 13, 17, 19, 37, 43, 67, 163\}$. This means that our proof is complete for all primes $p > 11$. 

    Now, let us consider the cases $p = 7$ and $p = 11$ to finish our proof.

    We have shown that $E(L)_{\text{tors}} = E(K)_{\text{tors}}$ when $p > 11$. By combining our result with the classification of torsion subgroups over quadratic fields given in Theorem \ref{najmanquadratic}, we have achieved the classification of $E(L)_{\text{tors}}$.

    Now, we will deal with the cases $p = 7$ and $p = 11$. In these two cases, the proof requires an extra step. Problems arise only in our usage of Lemma \ref{4pairinggeneralmore}. In the $p = 7$ case, we cannot use it to obtain $E(L)[43] = E(K)[43]$, so we need to handle it in another way. The same holds in the $p = 11$ case, where we cannot use the lemma to obtain $E(L)[67] = E(K)[67]$.

    Let $p = 7$ and let $L$ be a $\Z_p$-extension of $K$. Assume that $E(L)[43] \cong \Z/43\Z$. Then $E$ has a rational $43$-isogeny by Lemma \ref{n-isogeny}. Referring to Table 1 of \cite{Omer}, we can see that this is only possible when $E$ is the elliptic curve \lmfdbec{1849}{a}{1}, and the field of definition of the rational $43$-isogeny is $\Q(\zeta_{43})^+$. Clearly, this field cannot be contained in any $\Z_p$ extension of a quadratic field because its extension degree over $\Q$ is equal to $21$.

    Similarly, let $p = 11$ and let $L$ be a $\Z_p$-extension of $K$. Assume that $E(L)[67] \cong \Z/67\Z$. Then $E$ has a rational $67$-isogeny by Lemma \ref{n-isogeny}. Referring to Table 1 of \cite{Omer}, we can see that this is only possible when $E$ is the elliptic curve \lmfdbec{4489}{a}{1}, and the field of definition of the rational $67$-isogeny is $\Q(\zeta_{67})^+$. Clearly, this field cannot be contained in any $\Z_p$ extension of a quadratic field because its extension degree over $\Q$ is equal to $33$.

    With these last two cases, our proof is now complete.
\end{proof}

\section{Exceptional Cases}\label{edgecases}

In this section, we discuss the exceptional cases that were omitted in the statement of Theorem \ref{maintheorem}. Specifically, we focus on the cases where $p = 3$ and $p = 5$, presenting our results and insights regarding these cases.

\begin{theorem}\label{prime5complete}
Let $E/\Q$ be an elliptic curve. Let $K$ be a quadratic field. Let $L$ be a $\Z_5$-extension of $K$. Then, either $E(L)_{\text{tors}}\cong \Z/25\Z$ or
$E(L)_{\text{tors}}= E(K)_{\text{tors}}$. Moreover, if $K$ is an imaginary quadratic field, $E(K_{anti})_{\text{tors}}\not\cong \Z/25\Z$.     
\end{theorem}

\begin{proof}
    
The proof is very similar to the proof of Theorem \ref{maintheorem}. The only problem arises in our usage of Lemma \ref{4pairinggeneralmore}, because now we cannot use it to obtain $E(L)[11] = E(K)[11]$ or $E(L)[25] = E(K)[25]$. In order to deal with $11$-torsion, we will use isogenies. It is clear that $\zeta_{11} \not\in L$, so $E(L)$ cannot contain full $11$-torsion. Also, $E(L)[11^2] \cong \Z/11^2\Z$ is not possible, because $E$ cannot have a rational $121$-isogeny. Therefore, either $E(L)[11]$ is trivial or $E(L)[11] \cong \Z/11\Z$. If it is not trivial, then $E$ must have a rational $11$-isogeny. By using Table 1 of \cite{Omer}, we can see that this is only possible when $E$ is one of the elliptic curves \lmfdbec{121}{c}{2}, \lmfdbec{121}{b}{1}, or \lmfdbec{121}{a}{2}. In all three cases, the field of definition of the $11$-torsion is the field $\Q(\zeta_{11})^+$. Even though this is a quintic field, it cannot be contained in a $\Z_5$-extension of a quadratic field because $11$ is a ramified prime in $\Q(\zeta_{11})^+$, and the extension $L/K$ should be unramified outside $5$ due to Iwasawa \cite{iwasawa}. Therefore, $E(L)[11]$ must be trivial.

The other case is not as easy to handle. It is clear that $L$ does not contain $\zeta_5$, so $E(L)$ cannot contain full $5$-torsion. Moreover, $E(L)[5^\infty] \cong \Z/5^k\Z$ is not possible for $k \geq 3$, because $E$ cannot have a rational $125$-isogeny. If $E(L)[5]$ is trivial, there is no problem. If $E(L)[5^\infty] \cong \Z/5\Z$, we can still use Lemma \ref{4pairinggeneralmore} to obtain $E(L)[5^\infty] = E(K)[5^\infty]$. The problem arises when $E(L)[5^\infty] \cong \Z/25\Z$. In this case, $E$ must have a rational $25$-isogeny, but this is true for infinitely many elliptic curves $E/\Q$, because the modular curve $X_0(25)$ has genus $0$, unlike the genus $1$ modular curve $X_0(11)$.

Now, the special properties of the anti-cyclotomic extension $K_{anti}$ play an important role, and we can use Lemma \ref{anticyclomain} to show that if $E(K_{anti})[25] \cong \Z/25\Z$, then we have $E(K)[25] \cong \Z/25\Z$, but this is not possible due to Theorem \ref{najmanquadratic}.

In the case where $E(L)[5^\infty] \cong \Z/25\Z$, we can prove that $E(L)[q]$ is trivial for all other primes $q \neq 5$. The elliptic curve $E$ must have a rational $25$-isogeny, and therefore it cannot have a rational $q$-isogeny for any other prime $q$, as per Theorem \ref{rationalisogeny}. If $q > 5$ is a prime and $E(L)[q]$ is not trivial, then $E(L)[q] \cong \Z/q\Z$, because it cannot have full $q$-torsion due to the Weil pairing. In this case, $E$ has a rational $q$-isogeny, which is not possible due to our previous discussion.

Consider the $q = 3$ case, and observe that if $E(L)[3]$ is not trivial and $E$ does not have a rational $3$-isogeny, the only possibility is $E[3^\infty] \cong \Z/3\Z \times \Z/3\Z$. In this case, by using the same argument as in the proof of Lemma \ref{3torsionlemma}, we can see that $\Gal(\Q(E[3])/\Q)$ has order $2$, and it is easy to see that any subgroup of $GL_2(\Z/3\Z)$ of order $2$ must fix a subgroup of $E[3]$ of order $3$. Thus, $E$ must have a rational $3$-isogeny, which contradicts $E$ having a rational $25$-isogeny.

Consider the $q = 2$ case, and observe that if $E(L)[2]$ is not trivial and $E$ does not have a rational $2$-isogeny, then $E(L)$ must have full $2$-torsion. Since $L$ does not contain a cubic subfield, at least one of the points of order $2$ is a rational point, and hence $E$ has a rational $2$-isogeny, which contradicts $E$ having a rational $25$-isogeny. 

Thus, we have shown that if $E(L)[25] \cong \Z/25\Z$, then $E(L)_{\text{tors}} \cong \Z/25\Z$. In all other cases, we have $E(L)_{\text{tors}} = E(K)_{\text{tors}}$.
\end{proof}

\begin{remark}
  At this point, we do not know if $\Z/25\Z$ is realized as $E(L)_{\text{tors}}$ for any $Z_5$-extension $L$ of $K$. Rational elliptic curves with rational $25$-isogeny are classified as a $1$-parameter family and listed on the LMFDB database \cite{lmfdb}, ordered by conductor. Indeed, $\Z/25\Z$ can be a torsion subgroup over some quintic fields or fields of degree 10. In the LMFDB database, we checked all the elliptic curves with $25$-isogeny and conductor less than $1000$, and none of them has $25$-torsion over any $L$. It may or may not be realized, but our methods are not effective at proving a negative answer.
\end{remark}

\begin{theorem}\label{prime3cyc}
    Let $E/\Q$ be an elliptic curve. Let $K=\Q(\sqrt{d})$ be a quadratic field where $d$ is a square-free integer. Let $K_{cyc}$ be the cyclotomic $\Z_3$-extension of $K$. 
    Then,  $E(K_{cyc})_{\text{tors}}$ is one of the following groups:
     $$E(K_{cyc})_{\text{tors}} \simeq 
  \begin{cases}
      \mathbb{Z} / N \mathbb{Z} & \text{ with } 1\leq N\leq 10, \text{ or } N=12,13,15,16,18,21,27,  \\
           \mathbb{Z} / 2 \mathbb{Z} \times      \mathbb{Z} /2 N \mathbb{Z} & \text{ with } 1\leq N\leq 7,  \text{ or } N=9 \\
             \mathbb{Z} /3  \mathbb{Z}  \times    \mathbb{Z} /3 N \mathbb{Z} & \text{ with } N=1,2,3, \text{ only if } K=\mathbb{Q}(\sqrt{-3}) \text{ or}\\
              \mathbb{Z} /4  \mathbb{Z} \times    \mathbb{Z} /4 \mathbb{Z} & \text{ only if } K=\mathbb{Q}(\sqrt{-1}).   
  \end{cases}  $$ 
\end{theorem}

\begin{proof}
We will start with a different strategy than the proof of Theorem \ref{maintheorem}, because it is faster to eliminate non-realized groups from Corollary \ref{allgroupslist}. Since $K_{cyc}$ is an abelian Galois extension, we can benefit from its special properties. By a similar argument used in the proof of Lemma \ref{rootsofunityinZpextensions}, we can show that if $d \neq 1, 3$, then there are no imaginary roots of unity in $K_{cyc}$. Thus, with the Weil pairing, we can deduce that $E(K)_{cyc} \cong \Z/N\Z$ for some positive integer $N \leq 163$ or $E(K)_{cyc} \cong \Z/2\Z \times \Z/2N\Z$ for $1 \leq N \leq 9$. If $d = 1$, we get $K = \Q(\sqrt{-1})$, thus it is possible to have $E(K)_{cyc} \cong \Z/4\Z \times \Z/4N\Z$ for $N = 1, 2, 3, 4$, and $N$ cannot be larger due to Corollary \ref{allgroupslist}. If $d = 3$, then $K = \Q(\sqrt{-3})$, and we get $K_{cyc} = \Q(\mu_{3^\infty})$, so it contains infinitely many roots of unity, but we can see that it cannot contain full $9$-torsion due to Corollary \ref{allgroupslist}, and the only new possibilities we get are $E(K)_{cyc} \cong \Z/3\Z \times \Z/3N\Z$ for $N = 1, 2, 3$.

By using Lemma \ref{4pairinggeneralmore}, we can eliminate $\Z/N\Z$ for $N = 11, 25$. By using Lemma \ref{2torsionlemma4}, we can eliminate $\Z/2\Z \times \Z/16\Z$. In order to eliminate $\Z/N\Z$ for $N = 14, 19, 37, 43, 67, 163$, we need to consider isogenies. If we assume $E(K_{cyc})_{\text{tors}} \cong \Z/N\Z$, then by Lemma \ref{n-isogeny}, we can say that $E$ has rational $N$-isogenies. For the given values of $N$, there are only finitely many rational elliptic curves with rational $N$-isogenies, and the list of them can be found in Table 4 of \cite{fieldofdefinition}. Then, we can use the LMFDB database to find out the field of definition of the $N$-torsion to see that none of them are contained in $K_{cyc}$. As a final note, we can use Lemma \ref{4pairinggeneralmore} to obtain $E(K)_{\text{tors}} \cong \Z/15\Z$, in the case $E(K_{cyc})_{\text{tors}} \cong \Z/15\Z$. We know that this is possible only when $K = \Q(\sqrt{-15})$ or $K = \Q(\sqrt{5})$ by Theorem \ref{najmanquadratic}.

Let us consider $E(K)_{cyc} \cong \Z/4\Z \times \Z/4N\Z$ for $N = 1, 2, 3, 4$. We will show that $N = 2, 3, 4$ is not possible. Since $E(K_{cyc})[4] \cong \Z/4\Z \times \Z/4\Z$, we can say that $\Gal(\Q(E[4])/\Q)$ is a subgroup of $\Gal(K_{cyc}/\Q)$. Any finite subfield of $K_{cyc}$ has extension degree $3^a$ or $2 \cdot 3^a$ over $\Q$. From Galois representations, we know that $\Gal(\Q(E[4])/ \Q)$ is also a subgroup of $GL_2(\Z/4\Z)$. This group has order $96$. Thus, we can say that the order of $\Gal(\Q(E[4])/\Q)$ must divide $96$ and be a number of the form $2 \cdot 3^a$. This means that its order can be $1, 2, 3, 6$. It also needs to be even because $\Q(E[4])$ contains $\Q(\sqrt{-1})$ due to the Weil pairing. Thus, the only possibilities are $N = 2, 6$. We can use Magma \cite{magma} to see that $GL_2(\Z/4\Z)$ has subgroups of order $6$, but none of them are abelian groups, so they cannot be a subgroup of $\Gal(K_{cyc}/\Q)$, which is an abelian group. From this, we deduce that $\Gal(\Q(E[4])/\Q)$ must have order $2$, and $\Q(E[4]) = \Q(\sqrt{-1})$ holds. We have shown that full $4$-torsion is defined over $K$. If $E(K_{cyc})$ has non-trivial $3$-torsion, then we can use Lemma \ref{4pairinggeneralmore} to see that it is defined over $K$ too. Then, $E(K)_{\text{tors}}$ must contain a subgroup isomorphic to $\Z/4\Z \times \Z/12\Z$, which contradicts Theorem \ref{najmanquadratic}, so we just proved that $N \neq 3$. If $N = 2, 4$, then we get $E(K_{cyc})[8] \cong \Z/4\Z \times \Z/8\Z$. By using a $\Gal(K_{cyc}/\Q)$ orbit argument on the $x$-coordinates of the points of order $8$, we can show that there exists a point of order $8$ that belongs to $E(K_{cyc})_{\text{tors}}$. Combined with $E(K)[4] \cong \Z/4\Z \times \Z/4\Z$ that we showed earlier, we obtain $E(K)[8] \cong \Z/4\Z \times \Z/8\Z$, which contradicts Theorem \ref{najmanquadratic}. This proves that $N \neq 2, 4$, hence $N = 1$ is the only realized possibility.
\end{proof}

\begin{remark}
  In this case, we encounter many exceptional torsion structures that we did not encounter before. The case $E(K_{cyc})_{\text{tors}} \cong \Z/21\Z$ is realized only when $E$ is the elliptic curve \lmfdbec{162}{b}{1}. The imaginary quadratic field $K$ can be anything because the torsion subgroup is defined over $\Q(\zeta_9)^+$. The case $E(K_{cyc})_{\text{tors}} \cong \Z/27\Z$ is realized only when $E$ is the elliptic curve \lmfdbec{27}{a}{4}. Again, $K$ can be anything because the torsion is defined over $\Q(\zeta_{27})^+$. The case $E(K_{cyc})_{\text{tors}} \cong \Z/18\Z$ can be realized over infinitely many elliptic curves because the modular curve $X_0(18)$ has genus zero, so there are infinitely many elliptic curves with rational $18$-isogeny. In theory, it is possible, but we could not find any elliptic curves satisfying this in the LMFDB database.
\end{remark}

\begin{remark}
  In the case where $K = \Q(\sqrt{-3})$, our cyclotomic extension $K_{cyc}$ is equal to the field $\Q(\mu_{3^\infty})$, defined as $\Q$ adjoined by all the $3^a$ roots of unity. In this case, the torsion subgroup $E(\Q(\mu_{3^\infty}))_{\text{tors}}$ is classified by Guzvic and Krijan in \cite{guzvickrijan}, where they also proved that $E(\Q(\mu_{3^\infty}))_{\text{tors}} = E(\Q(\mu_{27}))$.

\end{remark}

\begin{theorem}\label{prime3anti}
    Let $E/\Q$ be an elliptic curve. Let $K=\Q(\sqrt{-d})$ be an imaginary quadratic field where $d$ is a square-free positive integer with $d\neq 1,3$. Let $K_{anti}$ be the anti-cyclotomic $\Z_3$-extension of $K$. 
    Then,  $E(K_{anti})_{\text{tors}}$ is one of the following groups:
     $$E(K_{anti})_{\text{tors}} \simeq 
  \begin{cases}
     \mathbb{Z} / N \mathbb{Z} & \text{ with } 1\leq N\leq 10, \text{ or } N=12,15,16,  \\
           \mathbb{Z} / 2 \mathbb{Z} \times      \mathbb{Z} /2 N \mathbb{Z} & \text{ with } 1\leq N\leq 7,  \text{ or } N=9.
  \end{cases}  $$ 
\end{theorem}

\begin{proof}
  This result is where the properties of the anti-cyclotomic extension make a great difference. In the previous result, we saw that multiple different torsion structures appeared when we worked with $\Z_3$-extensions. This is because we get sharper results when the extension degree is not divisible by $3$ while using Lemma \ref{4pairinggeneralmore}. We remedy this by using Lemma \ref{anticyclomain} to see that all the odd part of the torsion subgroup $E(K_{anti})_{\text{tors}}$ is realized over $K$. Now, we just have to analyze $E(K_{anti})[2^\infty]$. By Lemma \ref{2torsionlemma4}, we can see that $E(K_{anti})[2^\infty] = E(K)[2^\infty]$, except in the case $E(K_{anti})[2^\infty] \cong \Z/2\Z \times \Z/2\Z$. Other than this case, we obtain $E(K_{anti})_{\text{tors}} = E(K)_{\text{tors}}$, and we can use the classification in Theorem \ref{najmanquadratic}.

Now let us deal with the final case. Let $E(K_{anti})[2^\infty] \cong \Z/2\Z \times \Z/2\Z$, then we have $E(K_{anti})_{\text{tors}} \cong \Z/2\Z \times \Z/2N\Z$ for some odd positive integer $N$. In this case, by Lemma \ref{anticyclomain} and Theorem \ref{najmanquadratic}, we get $N = 1, 3, 5, 7, 9, 15$. All we have to do is to show that $E(K_{anti})_{\text{tors}} \cong \Z/2\Z \times \Z/30\Z$ is not possible. Let us assume that $E(K_{anti})_{\text{tors}} \cong \Z/2\Z \times \Z/30\Z$ holds. Then, by Lemma \ref{n-isogeny}, we can see that $E$ has a rational $15$-isogeny. Then, by using Table 4 in \cite{fieldofdefinition}, we can see there are only four possibilities for the elliptic curve $E$, and by looking at their torsion data in the LMFDB database, we can see that the base field must be $K = \Q(\sqrt{-15})$, but none of them satisfies $E(K_{anti})_{\text{tors}} \cong \Z/2\Z \times \Z/30\Z$. This concludes our classification.
\end{proof}

\begin{remark}
   In our classification, all groups except $\Z/2\Z \times \Z/14\Z$ and $\Z/2\Z \times \Z/18\Z$ can be realized as $E(K)_{\text{tors}}$ for some elliptic curve $E$ and a quadratic field $K$. However, the situation with the remaining two subgroups is unusual. If we choose the elliptic curve $E$ as \lmfdbec{54}{b}{3} and $K = \Q(\sqrt{-6})$, we obtain $E(K_{anti})_{\text{tors}} \cong \Z/2\Z \times \Z/18\Z$. On the other hand, we were unable to find any elliptic curve satisfying $E(K_{anti})_{\text{tors}} \cong \Z/2\Z \times \Z/14\Z$ in the LMFDB database.
\end{remark}

\begin{remark}
  When $K = \Q(\sqrt{-1})$, we can have $E(K_{anti})_{\text{tors}} \cong \Z/4\Z \times \Z/4N\Z$ for some integers $N$. If $N$ is odd, then we have shown that $N = 1, 3, 5, 7, 9$, because $E(K_{anti})[N] = E(K)[N]$. Moreover, by a $GL_2(\Z/4\Z)$ argument, we can show that $E(K_{anti})[4] = E(K_{anti}^{(1)})[4]$, where $K_{anti}^{(1)}$ is the first level of the anti-cyclotomic $\Z_3$-extension of $\Q(\sqrt{-1})$, which is the splitting field of the polynomial $x^3 - 3x - 4$ as seen in \cite{anticyclo}. Then, we obtain $E(K_{anti}^{(1)})_{\text{tors}} \cong \Z/4\Z \times \Z/4N\Z$. Since $K_{anti}^{(1)}$ is a sextic field, we can use Theorem 1 of \cite{sextic} to conclude that $N = 1, 3$. Moreover, if $N$ is even, we can show that $E(K_{anti}^{(1)})[8] \cong \Z/4\Z \times \Z/8\Z$ by using a $\Gal(K_{anti}/\Q)$ orbit argument. Again, this is not possible by the same result in \cite{sextic}.  
\end{remark}

\begin{remark}
    When $K = \Q(\sqrt{-3})$, we can have $E(K_{anti})_{\text{tors}} \cong \Z/3\Z \times \Z/3N\Z$ or $E(K_{anti})_{\text{tors}} \cong \Z/6\Z \times \Z/6N\Z$ for some integer $N$. If $N = 2^k 3^n m$ for some integers $k, n, m$, satisfying $\gcd(6, m) = 1$, we can see that $E(K_{anti})[m] = E(K)[m]$. We can use a $GL_2(\Z/3\Z)$ argument to show that $E(K_{anti})[3] = E(K_{anti}^{(1)})[3]$, where $K_{anti}^{(1)}$ is the first level of the anti-cyclotomic $\Z_3$-extension of $\Q(\sqrt{-3})$, which is the splitting field of the polynomial $x^3 - 3$ as seen in \cite{anticyclo}. Then, we can see that $E(K_{anti}^{(1)})_{\text{tors}}$ contains a subgroup isomorphic to $\Z/3\Z \times \Z/3m\Z$. Since $K_{anti}^{(1)}$ is a sextic field, we can use Theorem 1 of \cite{sextic} to conclude that $m = 1$. We have shown that $N = 2^k 3^n$ holds. In this case, we can also see that all the $2^k$-torsion is defined over $K$, by Lemma \ref{2torsionlemma4}. This means that $k = 0, 1, 2$ holds, as shown in \cite{sextic} in the $E(K_{anti})_{\text{tors}} \cong \Z/3\Z \times \Z/3N\Z$ case, and we get $N$ odd in the case $E(K_{anti})_{\text{tors}} \cong \Z/6\Z \times \Z/6N\Z$. In general, $N \in \{1, 2, 3, 4, 6, 8, 12, 16, 18\}$ must hold, just because $E$ has rational $N$-isogeny due to Lemma \ref{n-isogeny}.
\end{remark}


\begin{thebibliography}{99}

\bibitem{Omer} O. Avci, {\em Torsion of rational elliptic curves over the cyclotomic extensions of $\Q$}, submitted. 

\bibitem{magma} W. Bosma, J. Cannon, and C. Playoust, {\em The Magma algebra system. I. The user language}, J. Symbolic Comput.,
24 (1997), pp. 235-265.


\bibitem{anticyclo} D. Brink, {\em Prime decomposition on the anti-cyclotomic extension}, Math. Comp.
 {76} (2007), No. 260,  2127–2138.


\bibitem{Chou-p-adic} M. Chou, H.B. Daniels, I. Krijan, F. Najman, {\em Torsion groups of elliptic curves over the $\Z_p$-extensions of $\Q$}, New York J. Math. 27 (2021) 99–123.


\bibitem{Chou} M. Chou, {\em Torsion of rational elliptic curves over the maximal abelian extension of $\mathbb{Q}$}, Pacific J. Math. 302 (2019), No. 2, 481–509.





\bibitem{sextic}  T. Gu\v{z}vi\'{c},  {\em Torsion groups of elliptic curves over $\mathbb{Q}(\mu_{p^\infty})$}, J. Number Theory 220 (2021), 330-345.

\bibitem{guzvickrijan} T. Gu\v{z}vi\'{c}, I. Krijan, {\em Torsion groups of elliptic curves over some infinite abelian extensions of $\mathbb{Q}$}, submitted. 

\bibitem{guzvicvukorepa} T. Gu\v{z}vi\'{c}, B. Vukorepa, {\em Torsion groups of elliptic curves over $\mathbb{Q}(\mu_{p^\infty})$}, submitted. 

\bibitem{iwasawa}  K. Iwasawa, {\em On $\Z_\ell$-extensions of algebraic number fields}, Ann. of Math. (2) 98 (1973), 246–326. 



 \bibitem{lmfdb}The LMFDB Collaboration, {\em The L-functions and Modular Forms Database}, \url{http://www.lmfdb.org}, Online; accessed 25 December 2024.

\bibitem{fieldofdefinition} \'A. Lozano-Robledo, {\em On the field of definition of $p$-torsion points on elliptic curves over the rationals}, Mathematische Annalen, Vol 357, Issue 1 (2013), 279-305.

\bibitem{Mazur} B. Mazur, {\it Rational isogenies of prime degree}, Inventiones Math. 44 (1978), 129 - 162.


\bibitem{NajmanQuadratic} F. Najman, {\em Torsion of rational elliptic curves over cubic fields and sporadic points on $X_1(n)$}, Math. Res. Letters, 23 (2016) 245-272.


\end{thebibliography}
\end{document}